\documentclass[11pt]{amsart}

\newif\ifanonymous
\anonymousfalse 

\usepackage{amsmath,amssymb}
\usepackage[T1]{fontenc}
\usepackage{lmodern}
\usepackage[hidelinks]{hyperref}
\hypersetup{
  pdftitle={Every Reflection Group Is Natural},
  pdfsubject={Natural groups, right-invariant metrics, and reflection groups},
  pdfkeywords={natural group, reflection group, right-invariant metric,
    regular permutation group, threshold graph}
}
\ifanonymous
  \hypersetup{pdfauthor={}}
\else
  \hypersetup{pdfauthor={Alex J. Sutherland}}
\fi

\newtheorem{theorem}{Theorem}[section]
\newtheorem{proposition}[theorem]{Proposition}
\newtheorem{lemma}[theorem]{Lemma}
\theoremstyle{remark}
\newtheorem{remark}[theorem]{Remark}

\newcommand{\Isom}{\operatorname{Isom}}
\newcommand{\Aut}{\operatorname{Aut}}

\title[Reflection groups are natural]{Every Reflection Group Is Natural}
\ifanonymous
\else
  \author{Alex J. Sutherland}
  \address{Oregon State University}
\fi

\subjclass[2020]{Primary 20F55; Secondary 05C25, 20B07}
\keywords{Natural group, reflection group, right-invariant metric,
regular permutation group, threshold graph}

\begin{document}

\begin{abstract}
By a reflection group we mean an abstract group generated by nonidentity
involutions.  Knill proved that every reflection group of cardinality at most
the continuum is natural, in the sense that some metric determines its group
structure up to isomorphism among group structures whose right translations
are isometries, and asked whether the cardinality hypothesis can be removed.
We answer this question affirmatively.  More precisely, every reflection group
$G$ admits a right-invariant metric taking at most seven values for which the
full isometry group is the right-regular action of $G$.  The proof replaces
distinct numerical labels on generators by a finite-valued encoding of a
rigid graph on a well-ordered irredundant generating set.
\end{abstract}

\maketitle
\ifanonymous
  \pagestyle{myheadings}
  \markboth{REFLECTION GROUPS ARE NATURAL}{REFLECTION GROUPS ARE NATURAL}
\fi

\section{Introduction}

Following Knill~\cite{Knill}, a group $G$ is \emph{natural} if there is a
metric $d$ on its underlying set such that every group structure on that set
whose right translations are isometries is abstractly isomorphic to $G$.
Write $R(G)$ for the right-regular permutation group of $G$.  A particularly
strong way to establish naturality is to construct a right-invariant metric
such that
\[
  \Isom(G,d)=R(G).
\]

Knill proved that every reflection group of cardinality at most the continuum
is natural and asked whether the cardinality restriction is necessary
\cite{Knill,KnillProblems}.  Here, as in that question, a reflection group is
understood abstractly as a group generated by nonidentity involutions.  We
remove the restriction and obtain the stronger conclusion that the
naturalizing metric may be chosen to take only seven values, independently of
the cardinality of the group.

The cardinality issue is avoided by encoding an arbitrarily large generating
set through the pattern of finitely many distances rather than by assigning a
distinct real length to every generator.  We first choose a well-ordered
irredundant generating set.  A simple maximum-index argument then controls
collisions among products of two generators.  This permits us to encode a
rigidly marked threshold graph on the generator indices into a finite-valued
metric.

The closest antecedent is the canonical edge-colored complete Cayley graph:
the edge joining $g$ and $h$ receives the color
$\{gh^{-1},hg^{-1}\}$.  Byrne et al.~\cite{ByrneDonnerSibley}
observed, in their finite-group classification, that its color-preserving
automorphism group is regular whenever the group is generated by involutions.
Leemann and de la Salle~\cite[Theorem~2]{LeemannDeLaSalleRigid} subsequently
gave the corresponding classification for arbitrary groups.  Thus a fully
colored regular representation is already known in the present setting.  The
point here is the uniform finite coarsening: the full coloring can use one
color for every inverse pair in $G$, whereas our metric uses only six nonzero
distances and still has no isometries beyond the regular action.  Recent work
on color-permuting automorphisms of the complete Cayley graph may be found in
\cite{AlimirzaeiMorris}.

This distinction also separates the result from the literature on graphical
regular representations.  For finite groups, Babai classified the groups that
occur as Euclidean symmetry groups of vertex-transitive
polytopes~\cite{BabaiPolytopes}; his result includes all finite groups generated
by involutions.  He also proved that every infinite group has a directed
graphical regular representation~\cite{Babai}; for undirected Cayley graphs,
strong restrictions remain, and the finitely generated infinite case was
completed by Leemann and de la Salle~\cite{LeemannDeLaSalle}.  Finite-color
representations of permutation groups are naturally expressed in terms of
colored complete graphs; see, for example, Grech and
Kisielewicz~\cite{GrechKisielewicz}.  Standard background on finite threshold
graphs may be found in~\cite{MahadevPeled}, although the transfinite rigidity
argument needed here is proved directly.  We work in ZFC and use the
well-ordering theorem to order an involution generating set.

\begin{theorem}\label{thm:main}
Let $G$ be a group generated by involutions.  Then there is a right-invariant
metric
\[
  d\colon G\times G\longrightarrow [0,2]
\]
taking at most seven values such that
\[
  \Isom(G,d)=R(G).
\]
Consequently, every group generated by involutions is natural.
\end{theorem}

\section{A triangularly ordered involution generating set}

Let $R\subseteq G\setminus\{e\}$ be a set of involutions generating $G$.
Well-order $R$ and inspect its elements in that order.  Retain an element
precisely when it is not contained in the subgroup generated by the elements
retained earlier.  Reindex the retained elements in their inherited order as
\[
  S=\{s_\alpha:\alpha<\kappa\}.
\]

\begin{lemma}[Greedy triangularity]\label{lem:triangularity}
The set $S$ generates $G$, every $s_\alpha$ is an involution, and
\begin{equation}\label{eq:triangularity}
  s_\alpha\notin H_\alpha:=\langle s_\beta:\beta<\alpha\rangle
  \qquad(\alpha<\kappa).
\end{equation}
\end{lemma}

\begin{proof}
Property~\eqref{eq:triangularity} is exactly the retention rule.  Every
discarded element of $R$ belongs, at the moment it is discarded, to the
subgroup generated by already retained elements.  Hence every element of $R$
lies in $\langle S\rangle$.  Since $R$ generates $G$, so does $S$.
\end{proof}

For $\alpha>\beta$, define the unoriented two-generator product class
\[
  Q_{\alpha\beta}
  :=\{s_\alpha s_\beta,s_\beta s_\alpha\}
  =\{s_\alpha s_\beta,(s_\alpha s_\beta)^{-1}\}.
\]

\begin{lemma}[Maximum-index collision lemma]\label{lem:collision}
If $\alpha>\beta$, $\gamma>\delta$, and
\[
  Q_{\alpha\beta}\cap Q_{\gamma\delta}\ne\varnothing,
\]
then $\alpha=\gamma$.
\end{lemma}

\begin{proof}
Assume, for a contradiction, that $\alpha>\gamma$.  If the common element is
represented on the left as $s_\beta s_\alpha$, take inverses; the inverse of
either element of $Q_{\gamma\delta}$ remains in $Q_{\gamma\delta}$.  We are
therefore reduced to one of the equalities
\[
  s_\alpha s_\beta=s_\gamma s_\delta,
  \qquad
  s_\alpha s_\beta=s_\delta s_\gamma.
\]
Because all generators are involutions, these imply, respectively,
\[
  s_\alpha=s_\gamma s_\delta s_\beta,
  \qquad
  s_\alpha=s_\delta s_\gamma s_\beta.
\]
Since $\beta,\gamma,\delta<\alpha$, either equality places $s_\alpha$ in
$H_\alpha$, contradicting~\eqref{eq:triangularity}.  Thus $\alpha>\gamma$ is
impossible.  Interchanging the two pairs rules out $\gamma>\alpha$, so
$\alpha=\gamma$.
\end{proof}

\begin{lemma}[Separation of metric levels]\label{lem:separation}
If $\alpha\ne\beta$, then
\[
  s_\alpha s_\beta\notin\{e\}\cup S.
\]
\end{lemma}

\begin{proof}
The equality $s_\alpha s_\beta=e$ would imply $s_\alpha=s_\beta$.  Suppose
instead that $s_\alpha s_\beta=s_\gamma$.  If the three indices are distinct,
their maximum occurs only once in the equality and can be solved for as a
product of generators with smaller indices, contradicting
\eqref{eq:triangularity}.  If $\gamma=\alpha$, cancellation gives
$s_\beta=e$; if $\gamma=\beta$, cancellation gives $s_\alpha=e$.  Both are
impossible.
\end{proof}

\section{A marked rigid threshold graph on an ordinal}

Assume for now that $\kappa\geq 2$.  Define
$\epsilon\colon\kappa\to\{0,1\}$ by transfinite recursion:
\[
  \epsilon(0)=0,
  \qquad
  \epsilon(\alpha+1)=1-\epsilon(\alpha),
  \qquad
  \epsilon(\lambda)=0
\]
for every nonzero limit ordinal $\lambda<\kappa$.

Let $\Gamma_\kappa$ be the graph with vertex set $\kappa$ in which, for
$\alpha<\beta$,
\begin{equation}\label{eq:adjacency}
  \alpha\sim\beta
  \quad\Longleftrightarrow\quad
  \epsilon(\beta)=1.
\end{equation}
Thus each new vertex is adjacent either to every earlier vertex or to no
earlier vertex.

For a vertex $x$, write $N(x)$ for its open neighborhood and define the
graph-theoretic vicinal preorder by
\begin{equation}\label{eq:vicinal}
  x\preceq y
  \quad\Longleftrightarrow\quad
  N(x)\setminus\{y\}\subseteq N(y)\setminus\{x\}.
\end{equation}
Every graph automorphism preserves this preorder.

\begin{lemma}[Threshold rigidity]\label{lem:threshold}
For $\kappa\geq2$,
\[
  \Aut(\Gamma_\kappa)=\{\mathrm{id},(0\;1)\}.
\]
In particular, every automorphism fixing both $0$ and $1$ is the identity.
\end{lemma}

\begin{proof}
Let $\alpha<\beta$.  If $\epsilon(\beta)=1$, then every neighbor of $\alpha$
other than $\beta$ is also a neighbor of $\beta$, and hence $\alpha\preceq
\beta$.  If $\epsilon(\beta)=0$, direct inspection gives $\beta\preceq
\alpha$.  Thus $\preceq$ is a total preorder.

Let $x\equiv y$ mean that $x\preceq y$ and $y\preceq x$.  We claim that the
only nonsingleton equivalence class is $C=\{0,1\}$.  The vertices $0$ and $1$
are equivalent because they have identical adjacency to every vertex outside
$\{0,1\}$.

Fix $\alpha<\beta$ with $\{\alpha,\beta\}\ne\{0,1\}$.  If
$\epsilon(\alpha)=\epsilon(\beta)$, then $\beta\ne\alpha+1$, because
$\epsilon$ flips at every successor.  Thus $\alpha+1<\beta$, and the vertex
$\alpha+1$ is adjacent to exactly one of $\alpha$ and $\beta$.  Suppose
instead that $\epsilon(\alpha)\ne\epsilon(\beta)$.  If $\alpha>0$, any vertex
below $\alpha$ distinguishes $\alpha$ and $\beta$.  If $\alpha=0$, then
$\epsilon(\beta)=1$.  Since $\beta\ne1$, we have $2<\beta$, and the vertex
$2$ distinguishes $0$ and $\beta$.  This proves the claim.

Set
\[
  Z=\{\alpha\geq2:\epsilon(\alpha)=0\},
  \qquad
  O=\{\alpha\geq2:\epsilon(\alpha)=1\}.
\]
On the quotient by $\equiv$, the preorder is the linear order
\[
  Z^{\mathrm{op}}+C+O.
\]
Indeed, the order on $Z$ is the reverse of the inherited ordinal order, the
order on $O$ is the inherited ordinal order, every element of $Z$ precedes
$C$, and $C$ precedes every element of $O$.  The sets $Z$ and $O$, with their
inherited orders, are well-ordered.  A well-order and its reverse have no
nontrivial order automorphisms.  Since $C$ is the unique two-element
equivalence class, every graph automorphism fixes every singleton class and
preserves $C$ setwise.  Thus the only possible nonidentity automorphism
exchanges $0$ and $1$.  This exchange is indeed an automorphism.
\end{proof}

\section{The seven-valued metric}

Continue to assume that $\kappa\geq2$.  Define $\ell\colon G\to[0,2]$ by
\begin{align}
  \ell(e)&=0, \label{eq:length-e}\\
  \ell(s_0)&=1, \label{eq:length-s0}\\
  \ell(s_1)&=\frac98, \label{eq:length-s1}\\
  \ell(s_\alpha)&=\frac54 \qquad(2\leq\alpha<\kappa),
    \label{eq:length-generators}
\end{align}
and, for distinct $\alpha,\beta<\kappa$, by
\begin{equation}\label{eq:length-products}
  \ell(s_\alpha s_\beta)=
  \begin{cases}
    \dfrac32,&\epsilon(\max\{\alpha,\beta\})=0,\\[4pt]
    \dfrac74,&\epsilon(\max\{\alpha,\beta\})=1.
  \end{cases}
\end{equation}
Set $\ell(g)=2$ for every remaining $g\in G$.

\begin{lemma}[Well-definedness]\label{lem:well-defined}
The preceding cases define a function $\ell$ on $G$, and
\[
  \ell(g^{-1})=\ell(g)
  \qquad(g\in G).
\]
\end{lemma}

\begin{proof}
Lemma~\ref{lem:separation} shows that the three cases---identity, generator,
and product of distinct generators---are disjoint.  Suppose an element has two
representations as products of two distinct generators.  After ordering
each pair by its larger index, Lemma~\ref{lem:collision} shows that the two
maximum indices coincide.  Hence~\eqref{eq:length-products} assigns the same
value to both representations.

The identity and every generator are self-inverse.  The inverse of
$s_\alpha s_\beta$ is $s_\beta s_\alpha$, which has the same maximum index,
and the complement of the already assigned cases is inverse-closed.  Thus
$\ell(g^{-1})=\ell(g)$ for all $g$.
\end{proof}

Define
\begin{equation}\label{eq:metric}
  d(x,y):=\ell(xy^{-1}).
\end{equation}

\begin{lemma}\label{lem:metric}
The function $d$ is a right-invariant metric on $G$ taking values in
$\{0,1,9/8,5/4,3/2,7/4,2\}$.
\end{lemma}

\begin{proof}
Symmetry follows from Lemma~\ref{lem:well-defined}, and $d(x,y)=0$ holds
exactly when $x=y$.  For right invariance,
\[
  d(xg,yg)=\ell\bigl(xg(yg)^{-1}\bigr)=\ell(xy^{-1})=d(x,y).
\]
Every nonzero distance lies in $[1,2]$.  If $x,y,z$ are pairwise distinct,
then
\[
  d(x,z)\leq2\leq d(x,y)+d(y,z).
\]
The cases in which two of the three points coincide are immediate, so the
triangle inequality holds.
\end{proof}

\section{The full isometry group}

\begin{proposition}\label{prop:stabilizer}
Every isometry of $(G,d)$ fixing $e$ is the identity.
\end{proposition}

\begin{proof}
Let $F\in\Isom(G,d)$ satisfy $F(e)=e$.  By
\eqref{eq:length-s0}--\eqref{eq:length-generators}, the point $s_0$ is the
unique point at distance $1$ from $e$, the point $s_1$ is the unique point at
distance $9/8$, and the points at distance $5/4$ are exactly
$\{s_\alpha:2\leq\alpha<\kappa\}$.  Hence $F$ preserves $S$ and fixes
$s_0$ and $s_1$.

For distinct $\alpha,\beta$,
\[
  d(s_\alpha,s_\beta)=\ell(s_\alpha s_\beta),
\]
and the values $3/2$ and $7/4$ reproduce the adjacency relation
\eqref{eq:adjacency}.  Therefore the permutation induced by $F$ on $S$ is an
automorphism of $\Gamma_\kappa$.  It fixes $0$ and $1$, so
Lemma~\ref{lem:threshold} gives
\[
  F(s_\alpha)=s_\alpha
  \qquad(\alpha<\kappa).
\]

We now propagate this conclusion.  Suppose $F(g)=g$.  The points at distances
$1$, $9/8$, and $5/4$ from $g$ are exactly
\[
  s_0g,\qquad s_1g,\qquad\{s_\alpha g:2\leq\alpha<\kappa\},
\]
respectively.  Moreover,
\[
  d(s_\alpha g,s_\beta g)=\ell(s_\alpha s_\beta).
\]
Thus the same marked copy of $\Gamma_\kappa$ occurs around every $g$, and the
preceding argument shows that
\[
  F(g)=g
  \quad\Longrightarrow\quad
  F(s_\alpha g)=s_\alpha g
  \quad\text{for every }\alpha<\kappa.
\]
Starting at $e$ and inducting on finite word length in $S$, we conclude that
$F$ fixes every element of $G$.
\end{proof}

For $g\in G$, let $\rho_g(x)=xg^{-1}$.  Then
\[
  R(G):=\{\rho_g:g\in G\}
\]
is a faithful regular permutation representation of $G$.

\begin{proposition}\label{prop:isometry-group}
For the metric~\eqref{eq:metric}, one has $\Isom(G,d)=R(G)$.
\end{proposition}

\begin{proof}
Right invariance gives $R(G)\leq\Isom(G,d)$.  Given
$F\in\Isom(G,d)$, choose the unique $\rho_g\in R(G)$ for which
$\rho_g(e)=F(e)$.  Then $\rho_g^{-1}F$ fixes $e$, so
Proposition~\ref{prop:stabilizer} implies $\rho_g^{-1}F=\mathrm{id}$.
Hence $F=\rho_g$.
\end{proof}

\begin{proof}[Proof of Theorem~\ref{thm:main}]
If $\kappa\geq2$, use the metric constructed above.  If $\kappa=0$, then
$G$ is trivial.  If $\kappa=1$, then $G\cong C_2$, and the unique nonzero
distance on its two-point underlying set has full isometry group equal to its
regular action.

It remains to verify naturality.  Let $\star$ be any group structure on the
underlying set of $G$ for which all right $\star$-translations are isometries
of $d$.  Its right-translation group $K$ is a regular, hence transitive,
subgroup of
\[
  \Isom(G,d)=R(G).
\]
A subgroup of a regular permutation group is transitive only if it is the
whole regular group: for a fixed point, regularity makes each image correspond
to a unique element of $R(G)$.  Thus $K=R(G)$.  The group $(G,\star)$ is
abstractly isomorphic to its regular translation group, which is isomorphic to
the original group $G$.  Therefore $G$ is natural.
\end{proof}

\begin{remark}
The construction uses only seven metric values, including zero, independently
of the cardinality of $G$.  It therefore avoids assigning a distinct real
length to each reflection.  The metric is discrete and need not reflect any
pre-existing topology on $G$, as permitted by the abstract definition of
naturality in~\cite{Knill}.
\end{remark}

\ifanonymous
\else
\section*{Acknowledgments}

The author thanks Oliver Knill for helpful correspondence and for posing the
question that motivated this work.
\fi

\section*{Disclosure of AI-assisted tools}

The author used OpenAI's GPT-5.6 Sol during proof exploration and preliminary
drafting.  The author independently verified every argument and takes full
responsibility for the content of the manuscript.


\begin{thebibliography}{99}

\bibitem{AlimirzaeiMorris}
S.~Alimirzaei and D.~W.~Morris,
\emph{Colour-permuting automorphisms of complete Cayley graphs},
Art Discrete Appl. Math. \textbf{8} (2025), no.~1, Paper No.~P1.04, 18~pp.,
\url{https://doi.org/10.26493/2590-9770.1795.a62}.

\bibitem{Babai}
L.~Babai,
\emph{Infinite digraphs with given regular automorphism groups},
J. Combin. Theory Ser. B \textbf{25} (1978), no.~1, 26--46,
\url{https://doi.org/10.1016/S0095-8956(78)80008-2}.

\bibitem{BabaiPolytopes}
L.~Babai,
\emph{Symmetry groups of vertex-transitive polytopes},
Geom. Dedicata \textbf{6} (1977), no.~3, 331--337,
\url{https://doi.org/10.1007/BF02429904}.

\bibitem{ByrneDonnerSibley}
D.~P.~Byrne, M.~J.~Donner, and T.~Q.~Sibley,
\emph{Groups of graphs of groups},
Beitr. Algebra Geom. \textbf{54} (2013), no.~1, 323--332,
\url{https://doi.org/10.1007/s13366-012-0093-7}.

\bibitem{GrechKisielewicz}
M.~Grech and A.~Kisielewicz,
\emph{Abelian permutation groups with graphical representations},
J. Algebraic Combin. \textbf{55} (2022), 513--531,
\url{https://doi.org/10.1007/s10801-021-01060-8}.

\bibitem{Knill}
O.~Knill,
\emph{On graphs, groups and geometry},
arXiv:2205.14097v2 [math.GR] (2026),
\url{https://doi.org/10.48550/arXiv.2205.14097}.

\bibitem{KnillProblems}
O.~Knill,
\emph{2 open problems about natural groups},
Quantum Calculus, August 29, 2026,
\url{https://www.quantumcalculus.org/2-open-problems-about-natural-groups/}.

\bibitem{LeemannDeLaSalleRigid}
P.-H.~Leemann and M.~de la Salle,
\emph{Most rigid representation and Cayley index of finitely generated
groups},
Electron. J. Combin. \textbf{29} (2022), no.~4, Paper No.~P4.40, 9~pp.,
\url{https://doi.org/10.37236/10512}.

\bibitem{LeemannDeLaSalle}
P.-H.~Leemann and M.~de la Salle,
\emph{Cayley graphs with few automorphisms: the case of infinite groups},
Ann. Henri Lebesgue \textbf{5} (2022), 73--92,
\url{https://doi.org/10.5802/ahl.118}.

\bibitem{MahadevPeled}
N.~V.~R.~Mahadev and U.~N.~Peled,
\emph{Threshold Graphs and Related Topics},
Annals of Discrete Mathematics, vol.~56, North-Holland, Amsterdam, 1995.

\end{thebibliography}
\end{document}